\documentclass[12pt,english,reqno]{amsart}
\usepackage[T1]{fontenc}
\usepackage[]{inputenc}
\usepackage{amsthm}
\usepackage{amstext}
\usepackage{amssymb}
\usepackage{amsfonts}
\usepackage{graphicx,color}
\usepackage{subfig}

\usepackage[b]{esvect}
\usepackage{tikz}
\usetikzlibrary{arrows,calc,positioning,through,intersections,backgrounds,matrix,fit,shapes,patterns,decorations,decorations.pathmorphing,decorations.pathreplacing}

\makeatletter
\numberwithin{equation}{section}

\numberwithin{figure}{section}
\theoremstyle{plain}
\newtheorem{thm}{\protect\theoremname}
\newtheorem{conj}[thm]{\protect\conjecturename}
\newtheorem{cor}[thm]{\protect\corollaryname}

\theoremstyle{plain}
\newtheorem{lem}[thm]{\protect\lemmaname}
\newtheorem{prop}[thm]{\protect\propositionname}

\theoremstyle{definition}
\newtheorem{claim}{\protect\claimname}
\newtheorem*{unnumbered_claim}{\protect\claimname}

\providecommand{\definitionname}{Definition}
\usepackage{fullpage}
\usepackage{pdfsync}

\usepackage{enumitem}

\providecommand{\U}[1]{\protect\rule{.1in}{.1in}}

\makeatother

\usepackage{babel}
\providecommand{\claimname}{Claim}
\providecommand{\lemmaname}{Lemma}
\providecommand{\theoremname}{Theorem}
\providecommand{\conjecturename}{Conjecture}
\providecommand{\propositionname}{Proposition}
\providecommand{\corollaryname}{Corollary}
\providecommand{\N}{\mathbb{N}}
\providecommand{\E}{\mathbb{E}}

\newcommand{\tdeg}[2]{\|#1, #2\|}
\newcommand{\todeg}[2]{\|#1, #2\|^+}

\newcommand{\hdeg}[2]{\|#1, #2\|^h}
\newcommand{\ldeg}[2]{\|#1, #2\|^l}
\newcommand{\vdeg}[2]{\nu(#1, #2)}
\newcommand{\vodeg}[2]{\nu^+(#1, #2)}
\newcommand{\videg}[2]{\nu^-(#1, #2)}
\newcommand{\sdeg}[2]{\sigma(#1, #2)}
\newcommand{\sodeg}[2]{\sigma^+(#1, #2)}
\newcommand{\sideg}[2]{\sigma^-(#1, #2)}

\newcommand{\s}[1]{\mathcal{#1}}

\DeclareMathOperator{\im}{im}

\tikzset{
cc/.style={
rectangle,
draw=black,
minimum width=24,
minimum height=120
},
scc/.style={
cc,
minimum height=96
},
lcc/.style={
cc,
minimum height=144
},
vedge/.style={
->,
>=stealth',
very thick,
},
ccedge/.style={
->,
>=stealth',
thick,
out=270,
in=270,
looseness=2
},
bedge/.style={
thick,
decoration={
brace,
raise=.5
},
decorate
},
vertex/.style={
circle,
inner sep=1,
minimum size=8,
draw=black,
fill=black
}
}

\newcommand{\DrawNearlyEquitable}[4]{
	\def\a{#1}
	\def\b{#2}
	\def\ap{#3}
	\def\bp{#4}
        \pgfmathtruncatemacro\bm{\b - 1};
        \pgfmathtruncatemacro\bps{\b - \bp + 1};
        \pgfmathtruncatemacro\aps{\a - \ap + 1};

	\edef\defaultpgflinewidth{\the\pgflinewidth}
	
	\node (A1) [scc] {};
        \foreach \i in {2,...,\a} {
            \pgfmathtruncatemacro\j{\i - 1};
	    \node (A\i) [cc, below right=0 and -\defaultpgflinewidth of A\j.north east] {};
	}
	\node (B1) [cc, below right=0 and 0.2 of A\a.north east] {};
        \foreach \i in {2,...,\bm} {
            \pgfmathtruncatemacro\j{\i - 1};
	    \node (B\i) [cc, below right=0 and -\defaultpgflinewidth of B\j.north east] {};
	}
	\node (B\b) [lcc, below right=0 and -\defaultpgflinewidth of B\bm.north east] {};

	\draw [bedge] ([yshift=15pt]A\aps.north west) to node {$\mathcal{A}'$} ([yshift=15pt]A\a.north east);
	\draw [bedge] ([yshift=30pt]A1.north west) to node {$\mathcal{A}$} ([yshift=30pt]A\a.north east);
	\draw [bedge] ([yshift=15pt]B\bps.north west) to node {$\mathcal{B}'$} ([yshift=15pt]B\b.north east);
	\draw [bedge] ([yshift=30pt]B1.north west) to node {$\mathcal{B}$} ([yshift=30pt]B\b.north east);

	\node [above=1pt of A1.north] {$V^-$};
	\node [above=1pt of B\b.north] {$V^+$};
}

\begin{document}
\title{An extension of the Hajnal-Szemer\'edi theorem to directed graphs}

\author{Andrzej Czygrinow$^{*}$}\thanks{$^{*}$School of Mathematical Sciences and Statistics, Arizona
State University, Tempe, AZ 85287, USA. E-mail address: aczygri@asu.edu.}

\author{Louis DeBiasio$^{\dagger}$}\thanks{$^{\dagger}$Miami University, Oxford, OH 45056 USA. E-mail address: debiasld@miamioh.edu.}

\author{
H. A. Kierstead$^{\ddagger}$}\thanks{$^{\ddagger}$
School of Mathematical Sciences and Statistics, Arizona
State University, Tempe, AZ 85287, USA. E-mail address: kierstead@asu.edu.
Research of this author is supported in part by NSA grant H98230-12-1-0212.}

\author{Theodore Molla$^{\S}$}\thanks{$^{\S}$School of Mathematical Sciences and Statistics, Arizona
State University, Tempe, AZ 85287, USA. E-mail address: tmolla@asu.edu.
Research of this author is supported in part by NSA grant H98230-12-1-0212.}

 \begin{abstract}
 Hajnal and Szemer\'edi proved that every graph $G$ with $|G|=ks$ and $\delta(G)\ge\ k(s-1)$ contains $k$ disjoint $s$-cliques; moreover this degree bound is optimal.  We extend their theorem to directed graphs by showing  that every directed graph $\vv G$ with $|\vv G|=ks$ and $\delta(\vv G)\ge2k(s-1)-1$ contains $k$ disjoint transitive tournaments on $s$ vertices, where $\delta(\vv G)=\min_{v\in V(\vv G)}d^-(v)+d^+(v)$. Our result implies the Hajnal-Szemer\'edi Theorem, and its degree bound is optimal. We also make some conjectures regarding even more general results for multigraphs and partitioning into other tournaments. One of these conjectures is supported by an asymptotic result.
%HK A trivial, but in general best possible observation is that graphs with maximum degree at most $k-1$ have a proper $k$-coloring.  Hajnal and Szemer\'edi strengthened this observation in a remarkable way by showing that graphs with maximum degree at most $k-1$ have a proper $k$-coloring in which the sizes of the color classes are as equal as possible. We extend Hajnal and Szemer\'edi's theorem to directed graphs which in turn strengthens their original result.  We also make some conjectures regarding an even more general result for multigraphs.
 \end{abstract} 

\maketitle
\section{Introduction}
Let $G=(V,E)$ be a graph. An \emph{equitable} $k$\emph{-coloring}
of $G$ is a proper $k$-coloring whose color classes differ in size
by at most one. A factor of $G$ is a set $\mathcal{F}$ of disjoint subgraphs whose union spans $G$. The subgraphs in a factor are called \emph{tiles}. A factor $\mathcal{F}$
is an $H$-factor if each of its tiles is a copy 
of $H$. If $|G|=sk$, then the color classes of an equitable $k$-coloring
form a $\overline{K}_{s}$-factor. 
Our story starts in 1963:
\begin{thm}[Corr\'{a}di \& Hajnal 1963 \cite{corradi1963maximal}]
\label{thm:CH}Every graph $G$ with $|G|=n=3k$ and $\delta(G)\geq\frac{2}{3}n$
has a $K_{3}$-factor.\end{thm}

The following generalization was conjectured by Erd\H{o}s \cite{erdos} in 1963,
and proved  
%by Hajnal and Szemer\'{e}di \cite{HSz}
 seven years later: 
\begin{thm}[Hajnal \& Szemer\'{e}di 1970 \cite{hajnal1970pcp}]
\label{thm:comHS} Every graph $G$ with $|G|=n=ks$ and $\delta(G)\geq(1-1/s)n$
has a $K_{s}$-factor.
\end{thm} 

Since $|G|=\Delta(G)+\delta(\overline{G})+1$, Theorem~\ref{thm:comHS} has the following complementary form, in which
Hajnal and Szemer\'edi stated their proof of Erd\H{o}s' conjecture.
 
\begin{thm}[Hajnal \& Szemer\'{e}di 1970 \cite{hajnal1970pcp}] \label{HSthm}
 \label{theorem:HS}Every graph $G$ with $\Delta(G)\leq k-1$ has an
equitable $k$-coloring. 
\end{thm}
%It was in this form that Hajnal and Szemer\'edi stated their proof of Erd\H{o}s' conjecture.
%HK (which is perhaps a more natural way to state the result since the divisibility of $n$ can be swept under the definition of ``equitable'').  

The degree bounds in these theorems are easily seen to be tight. For example, $G:=K_{ks} -E(K_{k+1})$ satisfies $\delta(G)\ge (1-1/s)|G|-1$ but has no $K_s$-factor. 
The original proof of Theorem~\ref{thm:comHS} was quite involved, and
 only yielded an exponential
time algorithm. Short proofs yielding polynomial time algorithms appear in \cite{KK,KKY}; the following theorem provides a fast algorithm.
\begin{thm}[Kierstead, Kostochka, Mydlarz \& Szemer\'{e}di 2010 \cite{KKMS}]
\label{theorem:alg}Every graph $G$ on $n$ vertices with $\Delta(G)\leq k-1$
can be equitably $k$-colored in $O(kn^{2})$ steps. 
\end{thm}

In this paper  we consider extensions of Theorem~\ref{theorem:HS} for  \emph{simple} digraphs---those having no loops and  at most
two edges $xy,yx$ between any two vertices $x,y$. The \emph{in}- and \emph{out-degrees}
of a vertex $v$ are denoted by $d^{-}(v)$ and $d^{+}(v)$; the
\emph{total degree} of $v$ is the sum $d(v):=d^{-}(v)+d^{+}(v)$.
The \emph{minimum semi-degree} of $G$ is $\delta^0(G):=\min\{\min\{d^+(v), d^-(v)\}: v\in  V\}$ and the 
\emph{maximum semi-degree} of $G$ is $\Delta^0(G):=\max\{\max\{d^+(v), d^-(v)\}: v\in V\}$.
The \emph{minimum total degree} of $G$ is $\delta(G):=\min\{d(v):v\in  V\}$ and the 
\emph{maximum total degree} of $G$ is $\Delta(G):=\max\{d(v):v\in V\}$.
Among graphs on $p$ vertices, let $\vv T_p$ be the transitive tournament, $\vv C_p$ be the directed cycle, and $\vv K_p$ be the complete digraph with all possible edges in both directions. Let $E_G^+(X,Y)=\{xy\in E(G):x\in X\wedge y\in Y\}$.

The simplest way to extend Theorems~\ref{thm:CH} and \ref{thm:comHS} to digraphs is to replace minimum and maximum degree with minimum and maximum semi-degree.  However, there are two natural ways to weaken the semi-degree bounds and obtain a stronger result.  Instead of semi-degree, one can replace minimum and maximum degree by (i) minimum and maximum total degree or (ii) minimum and maximum out-degree (equivalently in-degree); cliques are replaced by (transitive) tournaments and independent sets are replaced by acyclic sets---much more on the reasons for these choices later. The case $s=3$ is completely solved by the following two theorems: %HK --- is correct, not --. For consistency, no spaces.
 
\begin{thm}[Wang 2000 \cite{wangdir}]
\label{thm:Wang}Every digraph $G$ with 
$\delta(G)\ge\frac{3|G|-3}{2}$ has $\left \lfloor |G|/3 \right \rfloor$ disjoint copies of $\vv C_3$.
Moreover, for
odd $k\in\mathbb Z^+$, the digraph $G:=\vv K_{3k}-E^+(X,Y)$, where
$X$ and $Y$ are disjoint and 
$|X|=|Y|+1=\frac{3k+1}{2}$
satisfies $\delta(G)=\frac{3|G|-5}{2}$, but 
does contain a $\vv C_3$-factor.
\end{thm} %HK replaced 3k by n. 

\begin{thm}
[Czygrinow, Kierstead \& Molla 2012 \cite{CKM12}] \label{cor:main}Suppose $G$
is a digraph with $|G|=n=3k$ and $\delta(G)\geq2\cdot\frac{2}{3}n-1$,
and $c\geq0$ and $t\geq1$ are integers with $c+t=k$. Then $G$
has a factor consisting of $c$ copies of $\vv C_3$ and $t$ copies of $\vv T_3$. Moreover, for every $k\in \mathbb Z^+$ the digraph $G:=\vv K_{3k}-E(\vv K_{k+1})$ satisfies $|G|=n$ and $\delta(G)=2\cdot\frac{2}{3}n-2$, but has no factor whose tiles are tournaments on three vertices.%HK replaced 3k by n. 
\end{thm}

Theorem~\ref{thm:Wang} gives an exact answer for cyclic $3$-tournaments, and Theorem~\ref{cor:main} with $t=k$ gives an exact answer for transitive $3$-tournaments. Moreover, it also shows that  the same bound $\delta(G)\geq4k-1$ also forces a factor with any combination of cyclic and transitive $3$-tournaments, except for $k$ cyclic tournaments. 

The extremal example for Theorem~\ref{cor:main} is the natural extension of the extremal example for Theorem~\ref{thm:comHS}; more generally the digraph $G:=\vv K_{sk}-E(\vv K_{k+1})$ satisfies $|G|=n=sk$ and $\delta(G)\geq 2(1-1/s)n-2$, but does not have a factor whose tiles are $s$-tournaments. %HK K_{sk}, not  K_{3k}.
The extremal example for Theorem~\ref{thm:Wang} seems to be an accident of small numbers---it works because it is not strongly connected. If $s\ge4$ and $\delta(G)\geq 2(1-1/s)n-2$, then $G$ is strongly connected.

Our main result is:

\begin{thm}\label{main}
Every digraph $G$ with $|G|=n=sk$ and $\delta(G)\ge 2(1-1/s)n-1$ has a $\vv T_s$-factor.
\end{thm}

We prove Theorem~\ref{main} in its following stronger complementary form by extending ideas developed in \cite{KK,KKore,KKMS,KKY}. An  \emph{equitable acyclic coloring} of a digraph is a coloring whose classes induce acyclic subgraphs
 (subgraphs with no directed cycles, including $2$-cycles), %HK For some reason some people do not talk about directed $2$-cycles
  and differ in size by at most one. 

\begin{thm}\label{mainc}
Every digraph $G$ with $\Delta(G)\leq 2k-1$ has an equitable acyclic $k$-coloring.
\end{thm}

To see that Theorem \ref{mainc} implies Theorem~\ref{main}, consider a digraph $G$ with $|G|=n=sk$ and $\delta(G)\ge 2(1-1/s)n-1$. Its complement $H$ satisfies $\Delta(H)\le 2n-2-(2(1-1/s)n-1)\leq 2k-1$. By Theorem~\ref{mainc}, $H$ has an equitable acyclic $k$-coloring. Since each color class is acyclic it can be embedded in a transitive $s$-tournament, whose complement is another transitive tournament %|G| in line 1, \le in line 2
contained in $G$.  Thus the tiles in $G$ induced by the color classes of $H$ contain transitive $s$-tournaments.

Even though a color class of an acyclic coloring may contain many edges, Theorem \ref{mainc} is  stronger than Theorem \ref{HSthm}.  To see this, let $G$ be a graph with $\Delta(G)\leq k-1$ and let $D$ be the graph obtained by replacing edge $uv$ of $G$ with two directed edges $uv$ and $vu$.  Since $\Delta(D)\leq 2k-2$, we may apply Theorem $8$ to obtain an equitable acyclic $k$-coloring.  Note that there are no edges in any color class since if $uv$ was in a color class, then $vu$ would be as well giving us a directed $2$-cycle.  Thus the equitable acyclic $k$-coloring of $D$ induces an equitable coloring of $G$.

% HK Simplified the language.  Since a color class of an acyclic coloring may contain many edges, it may be surprising to learn that Theorem \ref{mainc} is in fact stronger than Theorem \ref{HSthm}.  

The following two statements are neither implied by Theorem \ref{main} nor Theorem \ref{mainc} nor do they imply Theorem 
%HKadded another nor
 \ref{main} or Theorem \ref{mainc}.  However our proof can be slightly modified to give these results as well:
\begin{thm}~
\begin{enumerate}
\item Every digraph $G$ with $|G|=n=sk$ and $\delta^+(G)\ge (1-1/s)n$ has a $\vv T_s$-factor.

\item Every digraph $G$ with $\Delta^+(G)\leq k-1$ has an equitable acyclic $k$-coloring.
\end{enumerate}
\end{thm}

The paper is organized as follows.  In the remainder of this section  we introduce some more notation. In Section 2 we prove Theorem~\ref{mainc}. In Section 3 we introduce further conjectures concerning tiling with nontransitive tournaments, and generalizations to multigraphs. %HK +
In Section 4 we support one of these, Conjecture~\ref{conj:large_aclique_tiling}, by proving an asymptotic version.

\subsection{Notation}
For a digraph $G=(V,E)$ set $|G|=|V|$ and $\|G\|=|E|$. Also, set 
$E^+(X,Y)=E^-(Y,X)=\{xy\in E:x\in X \wedge y\in Y\}$, and $E(X,Y)=E^+(X,Y)\cup
E^-(X,Y)$. 
% TNM - added a few missing definitions (per AC comments)
Set $\|X,Y\| = |E(X,Y)|$, $\|X,Y\|^+ = |E^+(X,Y)|$ and 
$\|X,Y\|^- = |E^-(X,Y)|$. 
% TNM
An edge $e$ is \emph{heavy}  if it is contained in a  $2$-cycle; otherwise
it is \emph{light}. Let $\|X,Y\|^h$ denote the number of $2$-cycles contained in $E(X,Y)$. Then $2\|X,Y\|^h$ is the number of heavy edges in $E(X,Y)$. Let $\|X,Y\|^l$ denote the number of light edges in $E(X,Y)$.
We shorten $E(\{x\},Y)$ to 
%If $X=\{x\}$ or $Y=\{y\}$, we simply write $\left\Vert x,Y\right\Vert ^{+}$,
%$\left\Vert x,Y\right\Vert ^{-}$, 
$E(x,Y)$ and $E(X,V)$ to $E(X)$, etc.

\section{Main Result}%HK I added comments about an algorithm 
In this section we prove Theorem~\ref{mainc}. Our proof is based on the proof of Theorem~\ref{theorem:alg}. Although we do not go into the details, it also
 provides an $O(kn^2)$ algorithm.  Otherwise, our proof could be slightly simplified by avoiding the use of $\mathcal B'$.
 
 For simplicity, we shorten \emph{equitable acyclic} to \emph{good}. 
\begin{proof}[Proof of Theorem~\ref{mainc}]
  We may assume $\left\vert G\right\vert =sk$, where $s\in\mathbb{N}$:
  If $\left\vert G\right\vert =sk-p$, where $1\leq p<k$, then let
  $G^{\prime}$ be the disjoint union of $G$ and
  $\vv K_p$. %HK
  Then $\left\vert G^{\prime}\right\vert $
  is divisible by $k$, and $\Delta(G^{\prime})\leq 2k-1$, any good
  $k$-coloring of $G^{\prime}$ induces a good $k$-coloring
  of $G$. 

  Argue by induction on $\|G\|$. The base
  step $\left\Vert G\right\Vert =0$ is trivial; so suppose $u$ is
  a non-isolated vertex. Set $G^{\prime}:=G-E(u)$. By induction, $G^{\prime}$
  has a good $k$-coloring $f$. We are done unless some
  color class $U$ of $f$ contains a cycle $C$ with $u\in C$. Since
  $\Delta(G)\leq 2k-1$, for some class $W$ either $\left\Vert u,W\right\Vert ^{-}=0$
  or $\left\Vert u,W\right\Vert ^{+}=0$. Moving $u$ from $U$ to $W$ yields 
  an acyclic $k$-coloring of $G$ with all classes of size $s$, except
  % TNM moved the definition of $V^- and V^+ so we
  % can define V^+(f) and $V^-(f)$ % (per AC comments)
  for one \emph{small} class $U-u$ of size $s-1$ and one \emph{large}
  class $W+u$ of size $s+1$. Such a coloring is called a 
  \emph{nearly equitable} acyclic $k$-coloring. We shorten this to 
  \emph{useful}  $k$-coloring. %HK

  For a useful $k$-coloring $f$, let $V^- := V^-(f)$ be the small class
  and $V^+ := V^+(f)$ be the large class of $f$, and define an auxiliary
  % TNM 
  digraph $\mathcal{H}:=\mathcal{H}(f)$, whose vertices are the color classes,
  so that $UW$ is a directed edge if and only if $U\neq W$ and $W+y$ is acyclic
  for some $y\in U$. Such a $y$ is called a \emph{witness} for $UW$.
  If $W + y$ contains a directed cycle $C$, then we say that $y$
  is \emph{blocked} in $W$ by $C$. %HK
  If $y$ is blocked in $W$, then
  \begin{equation}
    \label{eq:blocked}
    \tdeg{W}{y} \ge 2.
  \end{equation}
  Let $\mathcal{A}$ be the set of classes that can reach $V^{-}$ in
  $\mathcal{H}$, $\mathcal{B}$ be the set of classes not in $\mathcal{A}$,
  and $\s{B'}$ be the set of classes that can be reached from $V^+$. %HK?????
  Call a class $W\in\mathcal{A}$ \emph{terminal}, if every $U\in\mathcal{A}-W$
  can reach $V^{-}$ in $\mathcal{H}-W$; so $V^{-}$ is terminal if
  and only if $\mathcal{A}=\{V^{-}\}$. Let $\mathcal{A}^{\prime}$
  be the set of terminal classes. A class in $\mathcal{A}$ with maximum
  distance to $V^{-}$ in $\mathcal{H}$ is terminal; 
  so $\mathcal{A}'\ne\emptyset$.
  For any $W \in V(\s{H})$ and any $x \in W$ we say $x$ is \emph{$q$-movable} 
  if it witnesses exactly $q$ edges in $E^+_{\s{H}}(W, \s{A})$.
  If $x$ is $q$-movable for $q \ge 1$, call $x$ \emph{movable}.
  Set 
  $a:=\left\vert \mathcal{A}\right\vert $, 
  $a^{\prime}:=\left\vert \mathcal{A}^{\prime}\right\vert $,
  $b:=\left\vert \mathcal{B}\right\vert $, 
  $b':= | \s{B'} |$, 
  $A:=\bigcup\mathcal{A}$,
  $A^{\prime}:=\bigcup\mathcal{A}^{\prime}$, 
  $B:=\bigcup\mathcal{B}$ and
  $B':=\bigcup\mathcal{B'}$.
   An edge  $e\in E(A,B)$ is called a crossing edge; denote its ends by $e_A$ and $e_B$, where $e_A\in A$.

  \begin{claim}
    \label{lem:sw}
    If 
%    $G$ has a useful $k$-coloring
%    $f$ with 
    $V^{+} \in \mathcal{A}$, 
    then $G$ has a good $k$-coloring.
  \end{claim}
  \begin{proof}
   
    Let $\mathcal{P}=V_{1}\ldots V_{k}$ be a $V^{+},V^{-}$-path in
    $\mathcal{H}$ .  Moving witnesses $y_{j}$ of $V_{j}V_{j+1}$ to $V_{j+1}$
    for all $j$ yields a good $k$-coloring of $G$.
  \end{proof}

%We now turn our attention to proving t
Establishing the next lemma completes the proof; notice the weaker degree condition.
  \begin{lem}
    \label{key}A digraph $G$ has a good $k$-coloring provided it has a useful  $k$-coloring
    $f$ with
     \begin{equation}
      \label{eq:deg_cond}
      d(v) \le 2k-1~ (= 2a + 2b - 1)~\textrm{for every vertex $v\in A^{\prime}\cup B$.}
    \end{equation}
   
%    $d(v)\leq 2k-1$ for every vertex $v\in A^{\prime}\cup B$
%    then $G$ has a good $k$-coloring. 
  \end{lem}
  \begin{proof}
    Arguing by induction on $k$, assume $G$
    does not have a good $k$-coloring.

%    For any color class $W$ of $f$, 
%    $x \in W$ and $y \in V \setminus W$ we say
%    $e \in \tedges{x}{y}$ is $W$-\emph{vital} if 
%    $e$ is on a directed cycle of $G[W + y]$.
A crossing edge $e$ with $e_A\in W\in \mathcal A$ is \emph{vital} if $G[W+e_B]$ contains a directed cycle $C$ with $e\in E(C)$. %HK REPLACED  $e_A$ and $e_B$ with $e$. They are equivalent.
 In particular if $xy$ is a crossing edge with $\|x,y\|=2$, then both $xy$ and $yx$ are vital.   For sets $S\subseteq A$ and $T\subseteq B$ denote the number of vital edges in $E(S,T)$, $E^-(S,T)$, and $E^+(S,T)$ by $\nu(S,T)$, $\nu^-(S,T)$, and $\nu^+(S,T)$, respectively. If $S=\{x\}$ or $T=\{y\}$, we drop the braces. Every $y\in B$ is blocked in $W$; so $\vodeg{W}{y}, \videg{W}{y} \ge 1$ and
    \begin{equation}
      \label{eq:vital_edges}
      %TM \tdeg{W}{y} \ge \vdeg{W}{y} \ge 2.
      \vdeg{W}{y} \ge 2.
    \end{equation}

    \begin{claim}
      \label{clm:edges_into_B}
      For any  $x \in W\in \mathcal{A}'$, if $x$ is $q$-movable, 
      then 
      \begin{equation*}
      \text{(a) }\|x, B\|\leq 2(b+q)+1-\|x, W\| \text{ and (b) }
\vdeg{x}{B} \le 2(b + q).
      \end{equation*}
%\begin{enumerate}      
% \item $\|x, B\|\leq 2(b+q)+1-\|x, W\|$ and
% \item $\vdeg{x}{B} \le 2(b + q)$.
%\end{enumerate}
    \end{claim}
    
    \begin{proof} %HK I did not like the formatting
   
   (a) There are $(a-1) - q$  classes in $\s{A} \smallsetminus W$ in which
      $x$ is blocked.
      So \eqref{eq:blocked} gives that
      %by \eqref{eq:vital_edges}, 
      $\tdeg{x}{A \smallsetminus W} \ge 2a - 2q - 2$.
      With \eqref{eq:deg_cond}, this implies
      \begin{equation*}
	\tdeg{x}{B} \le 2a + 2b - 1 - \tdeg{x}{A\smallsetminus W}-\tdeg{x}{W} 
	\le 2(b + q) + 1 - \tdeg{x}{W}.
      \end{equation*}
    
      (b) %By \eqref{eq:vital_edges} and 
      By (a), $\vdeg{x}{B}\leq 2(b + q) + 1 - \tdeg{x}{W}$.
      So the desired inequality holds if $\tdeg{x}{W} \ge 1$ or
      $\vdeg{x}{B}$ is even.
%      Clearly, the inequality holds if $\tdeg{x}{W} \ge 1$.
      If $\tdeg{x}{W} = 0$, then every vital edge incident to $x$ must be heavy.
      This implies that $\vdeg{x}{B}$ is even.
      % TM - from Louis' comment.
%      We are done, if $\vdeg{x}{B}$ is even; so assume it is odd.
%      For any $y \in B$, 
%      if $\tdeg{x}{y} = 2$ then $\vdeg{x}{y} = 2$; 
%      so there exists $y' \in B$ with $\vdeg{x}{y'} = \tdeg{x}{y'} = 1$.
%      Thus exactly one of $xy'$ and $y'x$ is in a cycle $C\subseteq W+y'$; the
%      other edge of $C$ incident to $x$ is in $G[W]$. Thus $\tdeg{x}{W} \ge 1$, and the claim follows.
%    \begin{enumerate}
%   \item  There are $(a-1) - q$ color classes in $\s{A} \smallsetminus W$ in which
%      $x$ is blocked.
%      Therefore,
%      by \eqref{eq:vital_edges}, 
%      $\tdeg{x}{A \smallsetminus W} \ge 2a - 2q - 2$.
%      With \eqref{eq:deg_cond}, this implies
%      \begin{equation*}
%	\tdeg{x}{B} \le 2a + 2b - 1 - \tdeg{x}{A} 
%	\le 2(b + q) + 1 - \tdeg{x}{W}.
%      \end{equation*}
%    
%    \item 
%      By \eqref{eq:vital_edges} and Claim \ref{clm:edges_into_B}(i), we have $\vdeg{x}{B}\leq 2(b + q) + 1 - \tdeg{x}{W}$.
%      If $\vdeg{x}{B}$ is even, then we are done, so
%      suppose $\vdeg{x}{B}$ is odd.
%      For any $y \in B$, 
%      if $\tdeg{x}{y} = 2$ then $\vdeg{x}{y} = 2$, 
%      so there exists $y' \in B$ such that $\vdeg{x}{y'} = \tdeg{x}{y'} = 1$.
%      So exactly one of $xy'$ and $y'x$ is in a cycle $C\subseteq W+y'$; the
%      other edge of $C$ incident to $x$ is in $G[W]$. Thus $\tdeg{x}{W} \ge 1$, and the claim follows.
%\end{enumerate}   
   \end{proof}

\begin{claim}\label{V-notterminal}
$V^-$ is not terminal.
\end{claim}

\begin{proof}
If $V^-$ is terminal, then $\s{A}=\{V^-\}$ and $a=1$; thus there are no movable
vertices.  Claim \ref{clm:edges_into_B}(b) implies $\vdeg{u}{B}\leq 2b$ for all $u\in A$ and \eqref{eq:vital_edges} implies $\vdeg{A}{w}\geq 2$ for all $w\in B$.  This yields the contradiction
\begin{equation*}
2(bs+1)=2|B|\leq \vdeg{A}{B}\leq 2b(s-1).
\end{equation*}
\end{proof}
%Suppose $V^-$ is terminal, which implies $\s{A}=\{V^-\}$ and $a=1$; thus there are no movable vertices.  Now Claim \ref{clm:edges_into_B}.(ii) implies $\vdeg{u}{B}\leq 2b(s-1)$ for all $u\in A$.  By \eqref{eq:vital_edges} we have $\vdeg{A}{w}\geq 2$ for all $w\in B$.  These facts together with \eqref{eq:order_of_sets} implies
%\begin{align*}
%2bs<2|B|\leq \vdeg{A}{B}\leq 2b(s-1),
%\end{align*}
%a contradiction.
%Let $V^-_0=\{u\in V^-:d(u)=2k-1 \text{ and } \|u, V^-\|=0\}$ and let $V^-_1=V^-\setminus V^-_0$.  Note that 
%\begin{equation}\label{V1toB}
%\|u, B\|\leq 2k-2, \text{ for all } u\in V^-_1.  
%\end{equation}
%Since $2k-1$ is odd, every vertex $u\in V^-_0$ is incident with at least one light crossing edge $uw$, which in particular implies 
%\begin{equation}\label{V0toB}
%\|V_0^-,B\|\geq |V_0^-|.  
%\end{equation}
%However, since $\|u, V^-\|=0$ for all $u\in V^-_0$, $u$ cannot lie on a cycle which blocks $w$ in $V^-$.  Thus
%\begin{equation}\label{BtoV0}
%\|w, V^-\|\geq 2+\|w, V^-_0\|, \text{ for all } w\in B. 
%\end{equation}
%Since $\|A, B\|=\|V^-_0, B\|+\|V_1^-,B\|$, \eqref{V1toB}, \eqref{V0toB}, and \eqref{BtoV0} give us
%\begin{align*}
%2((s-1)k+1)+|V_0^-|\leq \|A, B\|\leq (2k-1)|V^-_0|+(2k-2)(s-1-|V_0^-|).
%\end{align*}
%Which implies $2(s-1)k+2\leq 2(k-1)(s-1)$, contradicting the fact that $s\geq 1$.

Using Claim~\ref{lem:sw} and Claim \ref{V-notterminal}, $V^{+}\in\mathcal{B}$ and $\mathcal{A}\neq \mathcal{A'}$; thus
    \begin{equation}
      \label{eq:order_of_sets}
      \left\vert A\right\vert =as-1,
      |A'| = a's,
      \left\vert B\right\vert =bs+1,\text{ and }
      |B'| = b's + 1.
    \end{equation}

The next claim provides a key relationship between vertices in $A'$ and vertices in $B$.
    
    \begin{claim}%%HK Add paragraph of explanation.
      For all $x\in W \in \s{A}'$, and $y \in B$:
      \begin{enumerate} [label=(\alph*),nolistsep, ref={\theclaim(\alph*)}]
	\item
	  \label{clm:solo_not_movable}
	  if $G[W - x + y]$ is acyclic, then $x$ is not movable; and
	\item
	  \label{clm:distinct_solo_neighbors}
	 there is no $y' \in B' - y$ such that
	  $G[W - x + y + y']$ is acyclic.
      \end{enumerate}
    \end{claim}
    \begin{proof}
      By Claim~\ref{lem:sw} and Claim~\ref{V-notterminal}, 
      $W \notin \{V^-,V^+\}$.
      Suppose there exists $y \in B$ such that $G[W - x + y]$ is acyclic.  
      If there exists $y' \in B' - y$ such that
      $G[W - x + y + y']$ is acyclic,
      put $y_1 := y'$, $y_2 := y$ and $Y := \{y_1, y_2\}$;
      else put $y_1 := y$ and $Y := \{y_1\}$.
      Since $W \in \s{A}$, it contains a movable vertex.
      If $x$ is movable put $x' := x$; else
      let $x' \in W$ be any movable vertex; say $x'$ witnesses $WU$, where $U\in\s A$.
      Let $X := \{x', x\}$ and $W' := W \smallsetminus X + y_1$.

      Moving $x'$ to $U$ and switching witnesses along a $U,V^-$-path in $\s H-W$ yields
      a good $(a-1)$-coloring $f_{1}$ of        $G_{1}:=G[A\smallsetminus W+x']$.
      Also $f$
      induces a  $b$-coloring 
      $f_2$ of $G_2:=G[B - y_1]$.  
      It is good if $y_1\in V^+$; else it is  useful.  Since every $v \in B - y_1$ is blocked in
      every color class in $\s{A}(f)$,
       \eqref{eq:blocked} and \eqref{eq:deg_cond} imply
      $\Delta(G_2) \le 2k - 1 - 2a = 2b - 1$.
      By induction, there is  a good
      $b$-coloring $g_2$ of $G_2$. 
      (For algorithmic considerations, %HK??????
       note that if $y_1 \in B'$, as when $|Y| \ge 2$, then $g_2$ is immediately 
      constructible from $f_2$ using Claim~\ref{lem:sw},
      since then $%(V^+(f) =)~ 
      V^+(f_2) \in \s{A}(f_2)$.)

      If $|X| = 1$, then $|W'| = s$.
      So $g_1$, $W'$ and $g_2$ form a good $k$-coloring of $G$.
      This completes the proof of (a) (see Figure~\ref{fig:4a}).
      To prove (b), 
      suppose $|X| = |Y| = 2$.
      % TNM changed G_2 to B' - y_1 for clarity (per AC comment)
      It suffices to show that $G_3 := G[(B - y_1) + (W' + x)]$ has a good $(b+1)$-coloring. 
      % TNM
      By the case, $x$ is blocked in every class of $\s A \smallsetminus W$;
      so $\|x,A\smallsetminus W\|\ge2a-2$ by \eqref{eq:blocked}. Thus  $Z+x$ is acyclic for some class $Z\in \s B+W'$. So $G_3$ has a useful $(b+1)$-coloring $f_3$  with  
       $V^-(f_3)=W'$ and $V^+(f_3)=Z+x$, or  $Z=W'$ and $f_3$ is already good.  
   Since $W' + y_2$ is acyclic,
       $W' \notin \s{A}'(f_3) \cup \s{B}(f_3)$.
      By the definitions of $x$ and $B(f)$,
%TM      every $v \in V(G_3) \smallsetminus W'+x$
      every $v \in V(G_3) \smallsetminus W'$
      is blocked in every color class in $\s{A}(f) - W$.
      Thus, 
      by \eqref{eq:blocked} and \eqref{eq:deg_cond},
      $\tdeg{v}{V(G_3)} \le 2(b + 1) - 1$.
      So, by induction, there exists a good $(b+1)$-coloring $g_3$ of $G_3$
      (see Figure~\ref{fig:4b}).
    \end{proof} %HK You may want to try to improve my changes.

    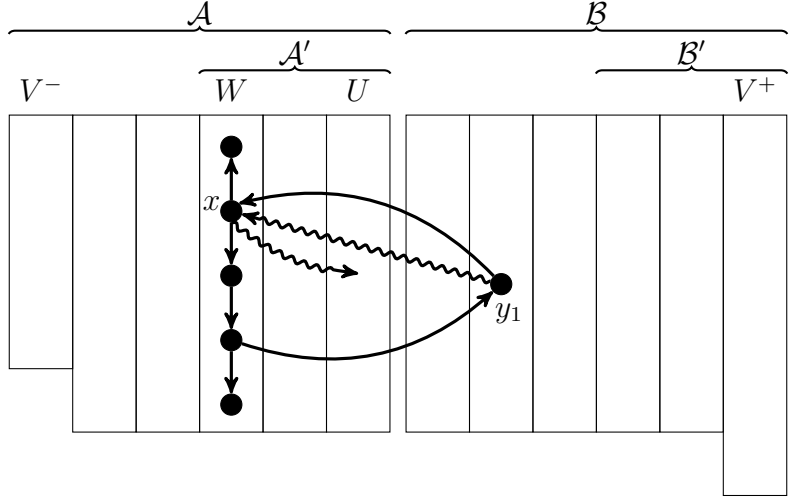
\begin{figure}
      \scalebox{1.0} {
        \begin{tikzpicture}[auto]

	  \DrawNearlyEquitable{6}{6}{3}{3}
	  % TNM - the edges of H seem to clutter the picture
%	  \DrawEdgesOfH

	  \node [above=1pt of A4.north] {$W$};
          \node [above=1pt of A6.north] {$U$};

	  \node (x1) [vertex,below=8pt of A4.north] {};
	  \node (x2) [vertex,below=16pt of x1] {};
	  \node [left=-3pt of x2.north west] {$x$};
	  \node (x3) [vertex,below=16pt of x2] {};
	  \node (x4) [vertex,below=16pt of x3] {};
	  \node (x5) [vertex,below=16pt of x4] {};

	  \draw (x2) [vedge] to (x1);
	  \draw (x2) [vedge] to (x3);
	  \draw (x3) [vedge] to (x4);
	  \draw (x4) [vedge] to (x5);

	  \node (y) [vertex,below=60pt of B2.north] {};
	  \node [below=0pt of y.south east] {$y_1$};
	  \draw [vedge,bend right] (y) to (x2.north east);
	  \draw [vedge,bend right] (x4) to (y);
	  \draw (x2.south) [vedge,decorate,decoration={snake,amplitude=1pt,segment
	  length=5pt,post length=8pt},bend right=15] to (A6.base); 
%	  \draw (y) [vedge,decorate,decoration={snake,amplitude=1pt,segment length=5pt,post length=6pt},in=-15,out=170] to (x2);
	  \draw (y) [vedge,decorate,decoration={snake,amplitude=1pt,segment length=5pt,post length=6pt}] to (x2);

	\end{tikzpicture}
	}
	\caption{
	After moving $x$ and $y_1$ as indicated, switching witness
	along a $U,V^-$-path in $\mathcal{H} - W$
	creates a good $a$-coloring of $G[A + y_1]$.
	By induction, there is a good $b$-coloring of $G[B - y_1]$.
	}
	\label{fig:4a}
      \end{figure}
      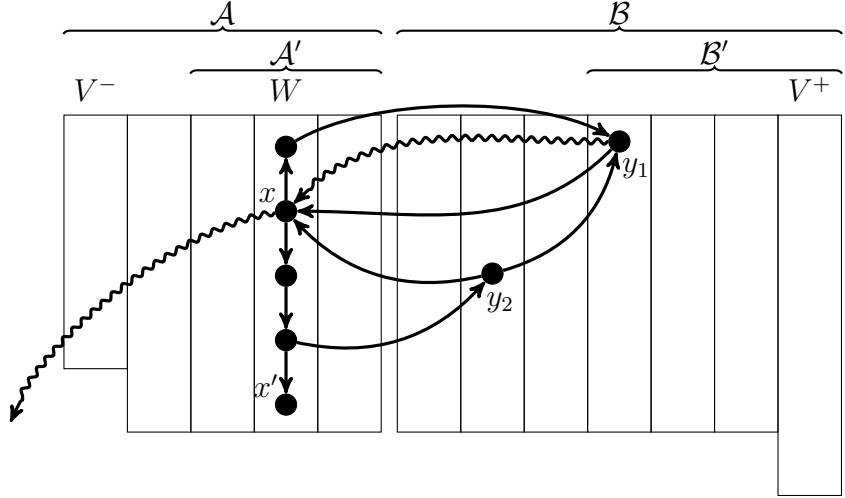
\begin{figure}
	\scalebox{1.0} {
	  \begin{tikzpicture}[auto]

	    \DrawNearlyEquitable{5}{7}{3}{4}
	  % TNM - the edges of H seem to clutter the picture
%	    \DrawEdgesOfH

	    \node [above=1pt of A4.north] {$W$};

	    \node (x1) [vertex,below=8pt of A4.north] {};
	    \node (x2) [vertex,below=16pt of x1] {};
	    \node [above left=-5pt of x2.north west] {$x$};
	    \node (x3) [vertex,below=16pt of x2] {};
	    \node (x4) [vertex,below=16pt of x3] {};
	    \node (x5) [vertex,below=16pt of x4] {};
	    \node [above left=-6pt of x5.north west] {$x'$};

	    \draw (x2) [vedge] to (x1);
	    \draw (x2) [vedge] to (x3);
	    \draw (x3) [vedge] to (x4);
	    \draw (x4) [vedge] to (x5);

	    \node (y) [vertex,inner sep=1pt] at (B2) {};
	    \node [below=0pt of y.south east] {$y_2$};
	    \node (yp) [vertex,below=6pt of B4.north] {};
	    \node [below right=-2pt and -3pt of yp.south] {$y_1$};

	    \draw [vedge,bend left] (y) to (x2);
	    \draw [vedge,out=225,in=0] (yp) to (x2);
	    \draw [vedge,bend right] (x4) to (y);
	    \draw [vedge,bend left,looseness=.7] (x1) to (yp);
	    \draw [vedge,bend right] (y) to (yp);

	    \draw (yp) [vedge,decorate,decoration={snake,amplitude=1pt,segment length=5pt,post length=6pt},in=45,out=180] to (x2);
	    \draw (x2) [vedge,bend right,looseness=.7,decorate,decoration={snake,amplitude=1pt,segment length=5pt,post length=6pt}] to ([xshift=-20pt,yshift=-20pt]A1.south west);
	  \end{tikzpicture}
        }
	\caption{
	After moving $x$ and $y_1$ as indicated, switching
	witnesses (one of which is $x'$)
	creates 
	a good $(a-1)$-coloring of $G[A - W + x']$
	and 
	a good $b$-coloring $g_2$ of $G[B - y_1]$.
	Placing $x$ in a color class of $g_2$ gives a useful 
	$(b+1)$-coloring of $G[B + W - x']$
	with small class $W' := W - x - x' + y_1$.
	By induction,
	there is a good $(b+1)$-coloring of $G[B + W - x']$
	because $G[W' + y_2]$ is acyclic. 
	}
	\label{fig:4b}
      \end{figure}

    A crossing edge $e\in E(W,B)$ is \emph{lonely} if it is vital and either  (i)  $e\in E^-(W,B)$ and $\nu^-(W,e_B)=1$ or (ii) $e\in E^+(W,B)$ and $\nu^+(W,e_B)=1$. If (i), then $e$ is \emph{in-lonely}; if (ii), then $e$ is \emph{out-lonely}.
 %HK CHANGES follow
    If $e$ is lonely, then  $G[W - e_A + e_B]$ is acyclic. For sets $S\subseteq A$ and $T\subseteq B$ denote the number of lonely, in-lonely and out-lonely edges in $E(S,T)$ by $\lambda(S,T)$, $\lambda^-(S,T)$  and $\lambda^+(S,T)$, respectively; drop braces for singletons. If $y\in B$, then $y$ is blocked in $W$. So
%    and, clearly,
    \begin{equation}
      \label{eq:solo_and_vital_edges}
      \vdeg{W}{y}+\lambda (W,y) \ge 4.
    \end{equation}

    \begin{claim}
      \label{clm:aprime_big}
      $a' > b$.
    \end{claim}
    \begin{proof}
      Assume $a' \le b$.
      Order $\mathcal{A}$ as $X_{1}:=V^{-},X_{2},\dots,X_{a}$ so that
      for all $j>1$ there exists $i<j$ with $X_{j}X_{i}\in E(\mathcal{H})$, and subject to this, order $\s{A}$ so that $l$ is maximum, where $l$ is the largest index of a non-terminal class. Set $W:=X_a$.

      The deletion of any non-terminal class leaves some class which can no longer reach $V^-$ in $\s{H}$; thus $l<a$, i.e., $W$ is terminal. Also $N^{+}_{\s{H}}(W)\subseteq\mathcal{A}'+X_{l}$,
      since otherwise we could increase the index $l$ by moving $W$ in
      front of $X_{l}$.  So if $x \in W$
      is $q$-movable, then 
      \begin{equation}\label{qa'}
	q \le a'.  
      \end{equation}
      If $\lambda (x,B) \ge 1$, then 
      % TNM added a little more explanation here
      there exists $y \in B$ such that 
      $\vodeg{x}{y} = 1$ or $\videg{x}{y} = 1$.
      In either case, $W - x + y$ is acyclic.  Therefore,
      % TNM
      Claim~\ref{clm:solo_not_movable} implies $q=0$; 
      since every lonely edge is vital, this and 
      Claim~\ref{clm:edges_into_B}(b) imply,
      $\lambda (x,B) \le \vdeg{x}{B} \le 2b$. 
      %HK (recall that $\sdeg{x}{B} \le \vdeg{x}{B}$ by definition).
      If $\lambda (x,B)=0$, then Claim~\ref{clm:edges_into_B}(b) and \eqref{qa'} gives
      $\vdeg{x}{B} \le 2(b + a') \le 4b.$
      Regardless, $\lambda (x,B) + \vdeg{x}{B}\leq 4b$. So 
      \begin{equation*}
	\lambda(W,B) + \vdeg{W}{B} 
	= \sum_{x \in W} \lambda(x,B) + \vdeg{x}{B} 
	\le 4b|W| \le 4bs.
      \end{equation*}
      This is a contradiction, since  
      \eqref{eq:solo_and_vital_edges} and \eqref{eq:order_of_sets} imply
      \begin{equation*}
	\lambda (W,B) + \vdeg{W}{B} 
	= \sum_{y \in B} \lambda (W,y) + \vdeg{W}{y}
	\ge 4|B| > 4bs.  \qedhere
      \end{equation*}
    \end{proof}

%    Now tighten the definition of a lonely edge. 
    A crossing edge $e\in E(W,B)$ is \emph{solo} if either  (i)  $e\in E^-(W,B)$ and $\|W,e_B\|^-=1$ or (ii) $e\in E^+(W,B)$ and $\|W,e_B\|^+=1$. If (i), then $e$ is \emph{in-solo}; if (ii), then $e$ is \emph{out-solo}.
 %HK CHANGES follow
% If $e$ is solo then it is lonely. 
    For sets $S\subseteq A$ and $T\subseteq B$ denote the number of solo, in-solo and out-solo edges in $E(S,T)$ by $\sigma(S,T)$, $\sigma^-(S,T)$  and $\sigma^+(S,T)$, respectively; drop braces for singletons. If $y\in B$, then $y$ is blocked in $W$. So
%    and, clearly,
    \begin{equation}
      \label{eq:solo}
      \|W,y\|+\sigma (W,y) \ge 4.
    \end{equation}

%    
%    
%    
%    
%    A crossing edge $e\in E(W,B)$ is \emph{in-solo} if $e\in E(W,e_B)$  and $\|W,e_B\|^-=1$; $e$ is $\emph{out solo}$ if  $e\in E(W,e_B)$  and $\|W,e_B\|^+=1$. For sets $S\subseteq A$ and $T \subseteq B$ let $\sigma^-(S,T)$ and $\sigma^+(S,T)$ denote the number of in-solo and out-solo edges in $E(S,T)$, respectively. Set $\sigma(S,T):=\sigma^-(S,T)+\sigma^+(S,T)$.

    Every $y \in B'$ is blocked in every color
    class in 
    $\s{A} \cup (\s{B} \smallsetminus \s{B}')$.
    So \eqref{eq:blocked} and \eqref{eq:deg_cond} give
    \begin{equation}\label{vitalA'}
      \|A',y\| 
      \le 2a + 2b - 1 - 
      \|A\smallsetminus A',y\|-\|B \smallsetminus B',y\| - \|y,B'\|\le 2a' + 2b' - 1 - \|y,B'\|.
    \end{equation}
    Using \eqref{eq:solo} and \eqref{vitalA'} we have
    \begin{align}
      \sdeg{A'}{y} &\ge \sum_{W \in \s{A}'} \left(4 - \tdeg{W}{y}\right)
      = 4a' - \tdeg{A'}{y} \notag \ge 2a' - 2b' + \tdeg{y}{B'} + 1 
      \\ &= 2(a' - b') + 2\hdeg{y}{B'} + \ldeg{y}{B'} + 1 \label{eq:solo_edge_count}.
    \end{align}

    Choose a maximal set $I$ subject to $V^{+}\subseteq I\subseteq B'$
    and $G[I]$ contains no $2$-cycle. 
    Let 
    $$J:=\{y\in I:\sdeg{A'}{y} = 2(a^{\prime}-b')+2\left\Vert y,B'\right\Vert ^{h}+1\}.$$
    Note that, by \eqref{eq:solo_edge_count}, the vertices in $J$ have the minimum possible number of solo-neighbors in $A'$ and additionally are incident with no light edges in $B'$.

    \begin{claim}
      \label{clm:solo_neighbors_from_I}
      Every $x \in A'$ satisfies $\sdeg{x}{I} \le 2$ .
      Furthermore, if there are distinct $y_1, y_2 \in I$
      such that $\sdeg{x}{y_1}, \sdeg{x}{y_2} \ge 1$, 
      then $\{y_1, y_2\} \subseteq I \smallsetminus J$.
    \end{claim}
    \begin{proof}
      Suppose $\sdeg{x}{I} \ge 3$ for some $x\in W \in \s A'$.
      By Claim~\ref{V-notterminal},  $W \neq V^-$.
      There exist distinct $y_1, y_2 \in I$ such that either
      $\sodeg{x}{\{y_1, y_2\}} = 2$ or $\sideg{x}{\{y_1, y_2\}} = 2$.
      Suppose $\sodeg{x}{\{y_1, y_2\}} = 2$. Then $\todeg{y_i}{W-x+y_i}=0$ for
      each $i\in [2]$. The choice of $I$ implies  $\|y_1,y_2\|\le1$. So there
      exists $i\in[2]$ with $\todeg{y_i}{W-x+y_1+y_2}=0$. Thus $G[W - x + y_1 + y_2]$ is acyclic,
      contradicting Claim~\ref{clm:distinct_solo_neighbors}.

      Now suppose there exist distinct $y_1\in J$ and $y_2 \in I$ with
      $\sdeg{x}{y_1}, \sdeg{x}{y_2} \ge 1$.
      By the definition of $J$ and \eqref{eq:solo_edge_count},
      $\ldeg{y_1}{B'} = 0$.
      Therefore, by the definition of $I$,
      $\tdeg{y_1}{I} = 0$ and in particular $\tdeg{y_1}{y_2} = 0$.
      So again $G[W - x + y_1 + y_2]$ is acyclic, 
      contradicting Claim~\ref{clm:distinct_solo_neighbors}.
    \end{proof}

    The maximality of $I$ implies that
    for all $y\in B \smallsetminus I$ 
    there exists $v\in I$ with $\|y,v\|=2$.
    Therefore, 
    \begin{equation}\label{ItoB}
      \sum_{y \in I} (2\hdeg{y}{B'} + 2) = 
      2\hdeg{B' \smallsetminus I}{I} + 2|I| \ge 
      2|B' \smallsetminus I| + 2 |I| = 2 |B'|.
    \end{equation}
    Also, by \eqref{eq:solo_edge_count} and the definition of $J$,
    for every $y \in I \smallsetminus J$,
    \begin{equation}\label{I-Jsolo}
      \sdeg{A'}{y} \ge 2(a' - b') + 2\hdeg{y}{B'} + 2.
    \end{equation}
    Therefore, by \eqref{I-Jsolo}, \eqref{ItoB}, \eqref{eq:order_of_sets}, Claim~\ref{clm:aprime_big},
    and the fact that $|I|\geq |V^+|> s$,
    \begin{align}
      \sdeg{A'}{I}+|J| &= \sum_{y\in I\setminus J}\sdeg{A'}{y}+\sum_{y\in J}(\sdeg{A'}{y}+1) \notag\\
      &\geq \sum_{y\in I}(2\left(a^{\prime}-b')+2 \hdeg{y}{B'} + 2\right) > 2s(a' - b') + 2|B'| > 2|A'|  \notag\\
%      \end{align*}
%Which we can write as
%\begin{equation} 
      \sdeg{A'}{I}&>2|A'|-|J|. \label{eq:solo_edges_lower_bound}
    \end{align}

    Claim \ref{clm:solo_neighbors_from_I} only gives $\sdeg{A'}{I}\leq 2|A'|$, so we have not reached a contradiction yet.  However, we will be saved by the fact that every vertex in $J$ forces at least one fewer solo edge between $A'$ and $I$.  Formally, let $A_1' := \{x \in A' : \sdeg{x}{I} \le 1\}$ and note that we can now write 
    \begin{equation}\label{-A1}
      \sdeg{A'}{I}\leq 2|A'|-|A_1'|.
    \end{equation}

    \begin{claim}\label{A1J}
      $|A_1'| \ge |J|$
    \end{claim}

    \begin{proof}
      For any $y \in J$, by the definition of $J$, $\sdeg{A'}{y}$ is odd.
      This implies that there exists $x \in A'$ such that $\sdeg{x}{y} = 1$.
      By Claim~\ref{clm:solo_neighbors_from_I}, 
      $\sdeg{x}{y'} = 0$ for all $y' \in I - y$.
      Therefore $x \in A_1'$.
    \end{proof}

    Finally by \eqref{eq:solo_edges_lower_bound}, \eqref{-A1}, and Claim \ref{A1J},
    \begin{equation*}
      2|A'| - |J| < \sdeg{A'}{I} \le 2|A'| - |A_1'|\leq 2|A'|-|J|,
    \end{equation*}
    a contradiction.  This completes the proof of Lemma~\ref{key}.
  \end{proof}

  Applying Lemma \ref{key} to the useful $k$-coloring $f$ completes the proof of Theorem \ref{mainc}.
\end{proof}

\section{Conjectures}

\global\long\def\mdeg#1#2{\|#1, #2\|}
%\global\long\def\ldeg#1#2{\|#1, #2\|^{l}}
\global\long\def\medges#1{\|#1\|}
\global\long\def\ledges#1{\|#1\|^{l}}
%\global\long\def\hdeg#1#2{\|#1, #2\|^{h}}
\global\long\def\eqy#1#2{Y_{#1}^{#2}}
\global\long\def\gey#1#2{Y_{#1}^{\ge#2}}
\global\long\def\ley#1#2{Y_{#1}^{\le#2}}

Removing the orientation from the edges of a directed graph $D$
leaves a loopless multigraph $M$ such that every
edge has multiplicity at most $2$. 
Call such a multigraph \emph{standard},
and say that $M(D)$ is the multigraph underlying $D$. %HK base???
For a fixed standard multigraph $M$, let $H(M)$ and $L(M)$ be the graphs on $V(M)$ containing
the edges of $M$ with multiplicity $2$ and $1$ respectively,
and put $G(M) := H(M) \cup L(M)$. If $M=M(D)$, then the edges of  $H(M)$ and $L(M)$ arise from the heavy and light edges of $D$, respectively; we extend this terminology to standard multigraphs. The multigraph $M$ is acyclic if it contains no cycles, including $2$-cycles. In other words, $G(M)$ is  acyclic and $\|H(M)\|=0$. If $G(M)$ is a complete graph we call $M$ a clique. 
The following conjecture implies Theorem~\ref{main}.
\begin{conj}
  \label{conj:meqclr}
  Every standard multigraph $M$ with $\Delta(M)\leq 2k-1$ has an equitable
  acyclic $k$-coloring. 
\end{conj}

We normally state Conjecture~\ref{conj:meqclr} in 
the following complimentary form. The complete standard multigraph $K^2_s$ on $s$ vertices is defined so that $H(K^2_s)=K_s$.  The \emph{complement}  of a standard multigraph $M$ is $\overline M:=K^2_{|M|}-M$, where $\mu_{\overline M}(xy)=2-\mu_M(xy)$. The complement of an acyclic standard multigraph on $s$-vertices is called a \emph{full $s$-clique}. 
\begin{conj}
  \label{conj:aclique_tiling}
  For every $s,k \in \N$,
  if $M$ is a standard multigraph
  on $sk$ vertices and $\delta(M) \ge 2(s-1)k - 1$,
  then $M$ contains $k$ disjoint full $s$-cliques.
\end{conj}

%In a forthcoming paper, we prove Conjecture~\ref{conj:aclique_tiling} 
%for large graphs.
%\begin{thm}
%  For every $s \in N$ there exists $k_0$ such that for any $k \ge k_0$,
%  if $M$ is a standard multigraph 
%  on $sk$ vertices and $\delta(M) \ge 2(s-1)k - 1$
%  then $M$ contains $k$ disjoint full $s$-cliques.
%\end{thm}

For the case $s=3$, Conjecture~\ref{conj:aclique_tiling} 
is a corollary of the main theorem in \cite{CKM12}.
We also make the following conjecture based on the work in \cite{CKM12}. 
\begin{conj}
  \label{conj:tri_strong}
  If $D$ is a strongly $2$-connected digraph on $3k$ vertices
  such that $\delta(D) \ge 4k - 1$,
  then $D$ has a cyclic triangle factor. 
\end{conj}
By slightly modifying the example from Theorem~\ref{thm:Wang}, 
we see that this is best possible:
For any $p\in \mathbb Z^+$,
let $k := 2p + 1$
and $D' := \vv K_{3k}-E^+(X,Y)$,  where $X$ and $Y$ are disjoint
and $|X|=3p + 2, |Y| = 3p + 1$.
To form the example $D$, 
reverse the orientation of the edges incident to some $v \in Y$.
$D$ is strongly connected 
and $\delta(D) = 9p + 2 = \frac{3|D| - 3}{2} - 1$ 
but $D$ has no cyclic triangle factor,
because all cyclic triangles in $D$
% TNM changed span to intersect (per AC comment)
that intersect $X$ and $Y$ contain $v$ and a vertex in $Y - v$.
% TNM

Let $K$ be an $s$-clique and let $\mathcal{D}$ be the set of all simple digraphs $D$ such that $K=M(D)$ (equivalently the set of all simple digraphs obtained by orienting the edges of $K$); we say $K$ is
\emph{universal} if for all $D\in \mathcal{D}$, $D$ contains every tournament on $s$ vertices. 
For example, the $3$-clique $Q=K^2_3-e$ on $3$ vertices with $5$ edges is universal---every orientation of $Q$ contains both
$\vv C_3$ and $\vv K_3$. Our goal is to factor standard multigraphs into universal tiles.

% TNM moved the definition of fit and near matching to section
% 4 - this paragraph quite a bit (per LD comments)
%Let $D$ be a digraph on $s$ vertices with
%underlying multigraph $M$.
Note that $K$ is universal if and only if 
for every tournament $T$ on $s$ vertices
and every orientation $D$ of $L(K)$ there
is an embedding of $D$ into $T$ (after embedding $D$ into $T$, every other
edge of $T$ corresponds to a heavy edge of $K$).
The following Theorem of Havet and Thomass\'{e}
and famous conjecture of Sumner, 
which has been proved for large values of $n$ \cite{KMOsumner11}, allow us to concisely say which cliques are universal.
%to formulate a conjecture similar to Conjecture~\ref{conj:tri_strong}
%for larger values of $s$.
%implies that acceptable cliques are universal.
\begin{thm}[Havet \& Thomass\'{e} 2000 \cite{ht00}]
  \label{thm:tourn_ham_paths}
  Every tournament
  $T$ on $n$ vertices contains every oriented path $P$ on 
  $n$ vertices except when
  $P$ is the anti-directed path and $n \in \{3, 5, 7\}$.
%  $n = 3$ and $T$ is the directed cycle;
%  $n = 5$ and $T$ is the regular tournament; or
%  $n = 7$ and $T$ is the Paley tournament.
\end{thm}
\begin{conj}[Sumner 1971]
  \label{conj:sumner}
  Every orientation of every tree on $n$ vertices is
  a subgraph of every tournament on $2n - 2$ vertices.
\end{conj}
With Theorem~\ref{thm:tourn_ham_paths},
we can state Conjecture~\ref{conj:sumner} in a form that
is more useful for our goal.
\begin{conj}
  \label{conj:sumner_forest}
  Let $T$ be a tournament on $n$ vertices and
  $F$ be a forest on at most $n$ vertices
  with $c$ non-trivial components. 
  If $F$ has at most $n/2 + c - 1$ edges,
  then $T$ contains every orientation of $F$.
\end{conj}
%With the following result, we will show that
%Conjecture~\ref{conj:sumner} is equivalent to 
%Conjecture~\ref{conj:sumner_forest}.
\begin{prop}
  \label{prop:sumner}
  Theorem~\ref{thm:tourn_ham_paths} and Conjecture~\ref{conj:sumner} 
  imply Conjecture~\ref{conj:sumner_forest}.
\end{prop}
%\begin{prop}
%  \label{prop:sumner}
%  Sumner's conjecture implies that
%  every acceptable $s$-clique with $s\ge4$ is a universal $s$-clique.
%\end{prop}
\begin{proof}
%  Let $M$ be an acceptable $s$-clique, and consider a digraph $D$  on $s$ vertices with $M\subseteq M(D)$. 
  Assume Conjecture~\ref{conj:sumner} is true.
  Let $D$ be an orientation of $F$.
%  First, assume that $M$ 
%  is fit; so $\|F\| \le \left\lfloor \frac{s}{2} \right\rfloor$.
We will argue by induction on $c$.
  Let $D_1$ be the largest component in $D$, 
  $D_2 := D - D_1$ and
  $m_i := \|D_i\|$ for $i \in \{1,2\}$.
  We can assume that $m_1 \ge 3$. Indeed, if
  $m_1 \le 2$, then $F$ is a collection of disjoint paths each on 
  at most $3$ vertices. Since $\|F\| \le 1$ when $n = 3$,
  Theorem~\ref{thm:tourn_ham_paths} implies that
  there is an embedding of  $D$ into $T$.

  Because there are $c - 1$ non-trivial components in $D_2$,
  $m_2 \ge c - 1$.
  Therefore,
  $m_1 \le n/2$ and, 
  since $D_1$ is a tree, $2|D_1| - 2 \le 2(n/2 + 1) - 2 = n$. 
  Conjecture~\ref{conj:sumner} then implies that 
  there is an embedding $\phi$ of $D_1$ into $T$.
  Note that this handles the case when $c = 1$.

  Let $T_2 := T - \phi(V(D_1))$.
  Since $m_1 \ge 3$ we have that 
  $m_1 \ge (m_1+1)/2 + 1$.
  Therefore,
  $\|D_2\| \le n/2 + c - 1 - m_1 \le (n - m_1 - 1)/2 + (c - 1) - 1$.
  Since $|T_2| = n - m_1 - 1$,
  there is an embedding of $D_2$ into $T_2$ by induction.
%  Because 
%  $\|D_1\| \le \left\lfloor \frac{s}{2} \right\rfloor$,
%  $|n_1| \le \left\lfloor \frac{s}{2} \right\rfloor + 1$.
%  Therefore, by Sumner's Conjecture, 
%  We can assume that $n_1 \ge 2$, so 
%  $\|D_1\| = n_1 - 1 \ge \left \lceil \frac{n_1}{2} \right \rceil$ and
%  $\|D_2\| \le 
%  \left\lfloor \frac{n}{2} \right\rfloor - \|D_1\| \le 
%  \left\lfloor \frac{n_2}{2} \right\rfloor$.
%
%  Now assume $M$ is a near matching, but not fit.
%  This implies that $s \ge 5$.
%  Therefore,
%  we can easily embed the largest oriented tree $F_1$ in $F$ into $T$,
%  because $|F_1| \le 3$.
%  Since all other components of $F$ have at most $1$ edge, it is 
%  trivial to extend this embedding to the remaining vertices of $F$.
\end{proof}

So if Sumner's conjecture is true, then universal $s$-cliques are those whose light edges induce a forest with $c$ non-trivial components and at most $s/2+c-1$ edges.  In light of this, we conjecture the following.
\begin{conj}
  \label{conj:large_aclique_tiling}
  For every $s \ge 4$ and $k \in \N$,
  if $M$ is a standard multigraph on
  $sk$ vertices with $\delta(M) \ge 2(s-1)k - 1$,
  then $M$
  can be tiled with $k$ disjoint universal $s$-cliques.
\end{conj}

In the following section we support Conjecture~\ref{conj:large_aclique_tiling} with  Theorems~\ref{thmS4}~and~\ref{thm:almost_tiling}.
Combining Theorem~\ref{cor:main}, 
Conjectures \ref{conj:tri_strong}, 
\ref{conj:sumner} and
\ref{conj:large_aclique_tiling} 
with Proposition~\ref{prop:sumner} we have the following.
\begin{conj}
  For any $s,k \in \N$,
  if $D$ is a strongly $2$-connected digraph on $sk$ vertices and 
  $\delta(D) \ge 2(s-1)k - 1$, then $D$ 
  contains any combination of $k$ disjoint tournaments on $s$ vertices.
\end{conj}

%Problem: Posa-Seymour for directed graphs with semi-degree?  If so Ghouil\'a-Houri proved that $\delta^0(G)\geq \frac{n}{2}$ implies that $G$ contains a directed hamiltonian cycle, which contains a $\vv{T}_2$-factor (when $G$ is strongly connected he proved the corresponding total degree result).

\section{An Asymptotic Result}

% TNM - moved definition of fit and near matching to hear also
% added the corollary (per AC and LD comments)
Let $K$ be a full clique on at most $s$ vertices. 
It is \emph{fit} if $\ledges K\le\max\{0,|K|-s/2\}$.
It is a \emph{near matching} if 
either $\ldeg{v}{K} \le 1$ for every vertex $v \in K$;
or $|K| = s$, $\ldeg{v}{K} \le 2$ for every vertex $v \in K$ 
and $\ldeg{v}{K} = 2$ for at most one vertex $v \in K$.
It is \emph{acceptable} if it is 
fit or a near matching.

In this section we prove the following theorem.
\begin{thm}\label{thmS4}
  For all $s \in \N$ and $\varepsilon > 0$
  there exists $n_0$ such that if
  $M$ is a standard multigraph on $n \ge n_0$ vertices, 
  where $n$ is divisible by $s$, then the following holds.
  If $\delta(M) \ge 2(1-1/s)n + \varepsilon n$, then
  there exists a perfect tiling of $M$ with acceptable $s$-cliques.
\end{thm}
If $K$ is acceptable and $|K| \ge 4$
then $L(K)$ is a forest with at most $|K|/2 + c - 1$ edges
where $c$ is the number of components of $L(K)$.
Therefore,
with Proposition~\ref{prop:sumner} and the
fact that Conjecture~\ref{conj:sumner} is true for large trees 
\cite{KMOsumner11}, we have the following corollary. 
\begin{cor}\label{thmS4cor}
  There exists $s_0$ such that 
  for any $s \ge s_0$ and any $\varepsilon > 0$
  there exists $n_0$ such that if
  $D$ is a directed graph on $n \ge n_0$ vertices, 
  where $n$ is divisible by $s$, the following holds.
  If $\delta(D) \ge 2(1-1/s)n + \varepsilon n$, then
  $D$ can be partitioned into tiles of order $s$ such that
  each tile contains every tournament on $s$ vertices.
\end{cor}
%TNM

First we show with Theorem~\ref{thm:almost_tiling} that for fixed $s$ we can
tile all but at most a constant number of vertices of $M$ with universal $s$-cliques. 

The following is a key step in the proof.

\begin{lem}
  \label{lem:improvement_lemma}
  Let $1\leq t\leq s-1$ and suppose $M$ is a standard multigraph. If $X_{1}$ and
  $X_{2}$ are fit $t$-cliques, $Y$ is a fit $s$-clique, and
  $\mdeg{X_{i}}Y\ge2(s-1)t+2-i$ for $i\in[2]$, then $M[X_{1}\cup X_{2}\cup Y]$
  contains two disjoint fit cliques with orders $t+1$ and $s$ respectively.  \end{lem}
  \begin{proof}
    Put $\eqy ic:=\{y\in Y:\mdeg{X_{i}}y=2t-c\}$
    and choose $x_1\in X_{1}$ with $\ldeg{x_1}{Y} \leq1$.

    Assume there exists 
    $y \in \eqy 10 \cup \eqy 20$ such that $\ldeg{y}{Y} \ge 1$.
    If $y \in \eqy 20$, then $Y - y + x_1$ and $X_2 + y$ are fit.
    If $y \in \eqy 10 \setminus \eqy 20$, then 
    % TNM inequality inverted (per AC comments)
    $X_1 + y$ is fit and $\mdeg{X_2}{Y - y} \ge 2(s-1)t - (2t - 1) = 2(s-2)t + 1$
    % TNM
    so there exists $x_2 \in X$ such that $\ldeg{x_2}{Y - y} \le 1$
    and $Y - y + x_2$ is fit.

    So we can assume $\ldeg{\eqy 10 \cup \eqy 20}{Y} = 0$.
    Since 
    \[
      |\eqy i0|+(2t-1)s\geq2|\eqy i0|+|\eqy i1|+(2t-2)s\geq ||X_i, Y||\geq 2(s-1)t+2-i,
      \]
      we have 
      \begin{align}
	\mbox{(a)}\,\,|\eqy i0| & \ge s-2t+2-i\mbox{ ~~~and~~~ (b)}\,\,|\eqy i0|+\frac{1}{2}|\eqy i1|\ge s-t+1-\frac{i}{2}.\label{eq:y0}
      \end{align}
      By (\ref{eq:y0}), 
      % TNM s and t were swapped here (per AC comments)
      if $t < s/2$ there exists $y_2 \in \eqy 20$
      and if $t \ge s/2$ there exists $y_2 \in \eqy 20 \cup \eqy 21$.
      % TNM
      Note that in either case $X_2 + y_2$ is fit.
      As $Y$ is full, $\alpha(L[\eqy 1 1])\geq\frac{1}{2}|\eqy 1 1|$. 
      So by (\ref{eq:y0}.b) there exists $I_{1}\subseteq \eqy 1 1$ such that
      $\ledges{I_1 \cup \eqy 1 0} = 0$ and $|I_{1} \cup Y_{1}^{0}| \ge s - t + 1$.
      Therefore we can select $Z_1 \subseteq I_{1} \cup Y_{1}^{0} - y_2$
      such that $|Z_1| = s - t$ and,
      by (\ref{eq:y0}.a) $|Z_1 \cap \eqy 1 0| \ge s - 2t$.
      $X_1 \cup Z_1$ is full and, 
      because $\ldeg{Z_1}{X_1} = |Z_1 \cap \eqy 1 1| \le \min\{t, s - t\}\leq s/2$,
      $X_1 \cup Z_1$ is fit.

    \end{proof}

    \begin{thm}
      \label{thm:almost_tiling}
      Let $s\geq 2$ and let $M$ be a standard multigraph on $n$ vertices.
      If $\delta(M) \ge 2(1-1/s)n - 1$, then
      there exists a disjoint collection of fit $s$-cliques
      that tile all but at most $s(s-1)(2s - 1)/3$ vertices of $M$.
    \end{thm}

    \begin{proof}
      Let $\mathcal{M}$ be a set of disjoint fit cliques in $M$, each having at most $s$ vertices.
      Let $p_i$ be the number of $i$-cliques in $\mathcal{M}$ and pick $\mathcal{M}$ so that
      $(p_{s}, \dotsc, p_1)$ is maximized lexicographically.
      Put $\mathcal{Y} := \{ Y \in \mathcal{M} : |Y| = s \}$
      and $\mathcal{X} = \mathcal{M} - \mathcal{Y}$.
      Set $U := \bigcup_{X \in \mathcal{X}} V(X)$,
      $W := \bigcup_{Y \in \mathcal{Y}} V(Y)$.
      Assume, for a contradiction, that $|U| > s(s-1)(2s - 1)/3$.  We claim that for all $X, X' \in \mathcal{X}$ with $|X| \le |X'|$, $\mdeg{X}{X'} \le 2\frac{s-2}{s-1}|X'||X|$. 

     If $X=X'$, then 
      $\mdeg{X}{X'}\le 2(|X|-1)|X| \le 2\frac{s-2}{s-1}|X|^2$. 
      
     If $X\neq X'$, then the maximality of $\mathcal{M}$ implies
      $x + X'$ is 
      not a fit $(|X'|+1)$-clique for any $x \in X$. Thus
      \begin{equation*}
	\mdeg{X}{X'} \le\begin{cases}
	  (2|X'| - 1)|X| = 2\frac{2|X'| - 1}{2|X'|}|X'||X| \le 2 \frac{s-2}{s-1}|X'||X| & \textrm{if } |X'| \le \frac{s-1}{2}; \\
	  (2|X'| - 2)|X| = 2\frac{|X'| - 1}{|X'|}|X'||X| \le 2 \frac{s-2}{s-1}|X'||X| & \textrm{if } \frac{s}{2} \le |X'| \le s-1.
	\end{cases}
%  If  $s/2 \le |X'| \le s-1$ then
%  $\mdeg{X}{X'} \le  (2|X'| - 2)|X| \le 2 \frac{s-2}{s-1}|X'||X|$.
%  If $|X'| \le \frac{s-1}{2}$ then 
%  $\mdeg{X}{X'} \le  (2|X'| - 1)|X| \le (2 - \frac{2}{s-1})|X'||X|$.
      \end{equation*}
      Therefore by the claim, 
      \begin{equation*}
	\mdeg{X}{U} 
	\le 2\frac{s-2}{s-1}|U||X| =
	2 \frac{s-1}{s}|U||X| - \frac{2}{s(s-1)}|U||X|
	< 2\frac{s-1}{s}|U||X| - |X|.
      \end{equation*}
      By the degree condition,
      \begin{equation}
	\label{eq:deg_to_W}
	\mdeg{X}{W} > 2\frac{s-1}{s}|W||X|.
      \end{equation}
      Since 
      \begin{equation*}
	\sum_{t = 1}^{s-1} 2t^2 = \frac{s(s-1)(2s-1)}{3} < |U| = \sum_{t=1}^{s-1} tp_t,
      \end{equation*}
      there exists $t \in [s-1]$ with $p_t \ge 2t + 1$.  
      Choose $\mathcal{X}' \subseteq \mathcal{X}$ 
      such that $|\mathcal{X}'| = 2t + 1$
      and $|X|=t$ for every $X \in \mathcal{X}'$.
      Put $U' := \bigcup_{X \in \mathcal{X'}}V(X)$.

      By \eqref{eq:deg_to_W},
      there exists $Y \in \mathcal{Y}$ such that
      $\mdeg{U'}{Y} \ge 2\frac{s-1}{s}|U'||Y| + 1 = 
      2(s-1)t|\mathcal{X'}| + 1$.
      % TNM added prime to script X - (per AC comments)
      Let $X_1, \dotsc, X_{2t+1}$ be an ordering of $\mathcal{X}'$
      % TNM
      such that 
      $\mdeg{X_i}{Y} \ge \mdeg{X_{i+1}}{Y}$ for $i \in [2t]$.
      Clearly, $2(s-1)t + 2t \ge \mdeg{X_1}{Y} \ge 2(s-1)t + 1$,
      so 
      \begin{equation*}
	\|U' - V(X_1), Y\| \ge 
	2(s-1)t(2t+1) + 1 - (2(s-1)t + 2t) = 
	(2(s-1)t - 1)2t + 1.
      \end{equation*}
      This implies $\mdeg{X_2}{Y} \ge 2(s-1)t$.
      Lemma~\ref{lem:improvement_lemma} applied to $X_1, X_2$ and $Y$ 
      then gives a contradiction to the maximality of $\mathcal{M}$.
    \end{proof}

    The next lemma,  adapted from an
    argument in \cite{levitt2010avoid}, is probabilistic.
    It requires
    the union bound, the linearity of expectation, Markov's inequality
    and Chernoff's inequality \cite{chernoff1952measure}. 
    For a $d$-tuple $T := (v_1, \dotsc, v_d) \in V^d$,
    let 
    $\im(T) := \{v_1, \dotsc, v_d\}$ denote the image of $T$.

% TNM Removed more precise versions of Chernoff's bound for simplicity
%\begin{thm}
%  Let $X$ be a random variable with binomial distribution. Then the
%  following hold for any $t\ge0$: 
%  \begin{enumerate}[label=(\alph*),nolistsep, ref={\theclaim~(\alph*)}]
%    \item $\Pr[X\ge\E[X]+t]\le\exp\left(-\frac{t^{2}}{2(\E[X]+t/3)}\right)$; 
%    \item $\Pr[X\leq\E[X]-t]\le\exp\left(-\frac{t^{2}}{2\E[X]}\right)$. 
%  \end{enumerate}
%  \label{chernoff} 
%\end{thm}

% TNM changed statement of Lemma - (per LD comments)
%    \begin{lem}
%      \label{thm:large_aclique_tiling_approx}
%      For any $m,d \in \N$, $\alpha > 0$, 
%      $\beta \in (0, \frac{\alpha}{2d})$ and
%      $\gamma \in 
%      \left(0, 2 \beta \left(\frac{\alpha}{2d} - \beta\right) \right)$
%      there exists $n_0\in \mathbb N$ such that 
%      for all sets $V$ with  $|V|=n \ge n_0$ and functions 
%      $f: \binom{V}{m} \to 2^{V^d}$ the following holds.
%      If $|f(S)| \ge \alpha n^d$ for every $S \in \binom{V}{m}$ then
%      there exists $\mathcal{F} \subset V^d$ such that:
%      \begin{enumerate} [label=(\alph*),nolistsep, ref={\theclaim(\alph*)}]
%	\item
%	  $d |\mathcal{F}| \le \beta n$; 
%	\item
%	  for every $T \in \mathcal{F}$ there exists $S \in \binom{V}{m}$ such that
%	  $T \in f(S)$;
%	\item
%	  $\im(T) \cap \im(T') = \emptyset$ 
%	  for every pair $T, T' \in \mathcal{F}$; and
%	\item
%	  $|f(S) \cap \mathcal{F}| \ge \gamma n$ for every $S \in \binom{V}{m}$.
%      \end{enumerate}
%      \label{lem:absorbing}
%    \end{lem}
\begin{lem}
  Let $m,d \in \N$, $\varphi > 0$,
  $\beta \in (0, \frac{\varphi}{2d})$ and
  $\gamma \in 
  \left(0, 2 \beta \left(\frac{\varphi}{2d} - \beta\right) \right)$.
  There exists $n_0$ such that 
  when $V$ is a set of order $n \ge n_0$ the following holds.
  For every $S \in \binom{V}{m}$, let $f(S)$ be a subset of $V^d$.
  Call $T \in V^d$ an \emph{absorbing tuple} if $T \in f(S)$ 
  for some $S \in \binom{V}{m}$.
  If $|f(S)| \ge \varphi n^d$ for every $S \in \binom{V}{m}$, then
  there exists a set $\mathcal{F}$ of at most $\beta n/d$
  absorbing tuples such that
  $|f(S) \cap \mathcal{F}| \ge \gamma n$ for every $S \in \binom{V}{m}$
  and the images of distinct elements of $\mathcal{F}$ are disjoint.
  \label{lem:absorbing}
\end{lem}
% TNM 
    \begin{proof}
      Pick $\varepsilon > 0$ so that  
      \begin{equation*}
	(1 + \alpha)\varepsilon <  \frac{\alpha \beta}{d} - 2\beta^2 - \gamma.
      \end{equation*}
      Let $\beta' := \frac{\beta}{d}$,
      $p := \beta' - \varepsilon$ and 
      $\gamma' := \gamma + (d^2 + 1)p^2$.
      Let $\mathcal{F}'$ be a random subset of $V^d$ where each $T \in V^d$
      is selected independently with probability $p n^{1 - d}$.
      Let 
      \begin{equation*}
	\mathcal{O} := 
	\left\{ 
	  \{T, T'\} \in \binom{V^d}{2} : \im(T) \cap \im(T') \neq  \emptyset 
	\right\}
      \end{equation*}
      and $\mathcal{O}_{\mathcal{F}'} := \mathcal{O} \cap \binom{\mathcal{F}'}{2}$.

      We only need to show that, for sufficiently large $n_0$,
      with positive probability 
      $|\mathcal{O}_{\mathcal{F}'}| < (d^2 + 1)p^2 n$,
      $|\mathcal{F}'| < \beta' n$ and
      $|f(S) \cap \mathcal{F}'| > \gamma' n$ for every $S \in \binom{V}{m}$.
      Indeed, we can then remove at most $(d^2 + 1)p^2 n$ tuples
      from such a set $\mathcal{F}'$ so that the images of the remaining
      tuples are disjoint.
      The resulting set will satisfy (a), (c) and (d).
      To also satisify (b), 
      also remove every 
      $T \in \mathcal{F}'$ for which 
      there does not exist $S \in \binom{V}{m}$
      such that $T \in f(S)$.

      Clearly,  
      \begin{equation*}
	|\mathcal{O}| \le n \cdot d^2 \cdot n^{2d-2}=
	d^2 n^{2d - 1},
      \end{equation*}
      so for any $\{T, T'\} \in \binom{V^d}{2}$,
      $\Pr(\{T, T'\} \subset \mathcal{F}') = p^2 n^{2 - 2d}$.
      Therefore, by the linearity of expectation,
      $\E[|\mathcal{O}_{\mathcal{F}'}|] < d^2 p^2 n$.
      So, by Markov's inequality, 
      \begin{equation*}
	\Pr\left( 
	\left|\mathcal{O}_{\mathcal{F}'} \right| \ge (d^2 + 1) p^2 n 
	\right) \le 
	\frac{d^2}{d^2 + 1}.
      \end{equation*}
      Note that $\E[|\mathcal{F}'|] = p n$ and 
      $p n \ge \E[|f(S) \cap \mathcal{F}'|] \ge \alpha p n$ for every $S \in \binom{V}{m}$.
      Therefore, by Chernoff's inequality,
% TNM Removed more precise versions of Chernoff's bound for simplicity
%  Therefore, by Theorem~\ref{chernoff},
%  \begin{equation*}
%    \Pr(|\mathcal{F}'| \ge \beta' n) \le 
%    \exp\left(\frac{-3\varepsilon^2 n}{6p + 2\varepsilon}\right),
%  \end{equation*}
      $
      \Pr(|\mathcal{F}'| \ge \beta' n) \le 
      \exp(-\varepsilon^2 n/3),
      $
      and, since 
      \begin{equation*}
	\alpha p - \gamma' =
	\frac{\alpha \beta}{d} - \alpha \varepsilon - 
	(d^2+1)\left(\frac{\beta}{d} - \varepsilon\right)^2 - \gamma >
	\frac{\alpha \beta}{d} - 2 \beta^2 - \gamma - \alpha \varepsilon
	> \varepsilon,
      \end{equation*}
%  $\Pr(|\mathcal{F}' \cap f(S)| \le \gamma' n) \le 
%  \exp\left(-\frac{\varepsilon^2 n}{2p} \right)$
      $\Pr(|\mathcal{F}' \cap f(S)| \le \gamma' n) \le 
      \exp(-\varepsilon^2 n/3)$
      for every $S \in \binom{V}{m}$.
      Therefore, for sufficiently large $n_0$,
      \begin{equation*}
	\Pr\left( 
	\left|\mathcal{O}_{\mathcal{F}'} \right| \ge (d^2 + 1) p^2
	\right) 
	+
	\Pr(|\mathcal{F}'| \ge \beta' n) 
	+
	\sum_{S \in \binom{V}{m}} \Pr(|\mathcal{F}' \cap f(S)| \le \gamma' n)
	< 1.
	\qedhere
      \end{equation*} 
    \end{proof}

    % TNM restricted lemma to s \ge 2 
    % and fixed proof when t = 0 (per AC comment) 
    \begin{lem}
      Let $s \ge 2$, $\varepsilon > 0$, and $M = (V, E)$ be a standard multigraph on $n$ vertices.  
      If $\delta(M) \ge 2\frac{s-1}{s}n + \varepsilon n$, then for all distinct $x_1, x_2 \in V$,
      there exists $A \subseteq V^{s-1}$ such
      that $|A| \ge \left( \varepsilon n \right)^{s-1}$
      and for every $T \in A$ 
      both $\im(T) + x_1$ and $\im(T) + x_2$ 
      are near matching $s$-cliques.
      \label{lem:sponges}
    \end{lem}
    \begin{proof}
      For $0 \le t \le s-1$, the $t$-tuple 
      $T \in V^t$ is called  \textit{useful} if, for both $i \in \{1, 2\}$,
      $\im(T) + x_i$ is a near matching
      %, $\ldeg{v}{\im(T) + x_i} \le 1$ for every $v \in \im(T)$
      %isn't this redundant since a vertex of degree 2 cannot exist before x_i is added? LD
      % TNM - removed ')' (per AC comments)
      and $\ldeg{x_i}{\im(T)} \le \max\{0, t - s + 3\}$.
      % TNM
      To complete the proof we will show that there
      exists a set $A \subseteq V^{s-1}$ such that
      $|A| \ge (\varepsilon n)^{s-1}$ 
      and every $T \in A$ is  useful. %$(s-1)$-tuple.
      Suppose there is no such set.

      Let $0 \le t < s-1$ be the maximum integer for which 
      there exists $A \subseteq V^t$
      such that $|A| \ge (\varepsilon n)^{t}$
      and every $T \in A$ is useful (note that $0$ is a candidate since the empty function $f :\emptyset \to V$ is useful and $V^0 = \{ f \}$) ; %It exists as  $0$ is a candidate.
      select $A$ so that in addition 
      $\sum_{T \in A} \medges{M[\im(T)]}$ is maximized.
      Since $t$ is maximized, there exists 
      $(v_1, \dotsc, v_t) \in A$ 
      with  less than $\varepsilon n$ extensions
      $(v_1, \dotsc, v_t, v)$ to a useful $(t+1)$-tuple. 

      Let $m := n/s$, 
      $Y := \{x_1, x_2,v_1,\dots,v_t\}$, and 
      $V_c := \{ v \in V: \mdeg{v}{Y} \ge 2t + 4 - c\}$.
      Then $|V_0| \le \varepsilon n$, since each $v\in V_0$ extends $(v_1 ,\dots, v_t)$.
      Define 
      \begin{equation*}
	Z := 
	\begin{cases} 
	  V_1 \cap N_H(x_1) \cap N_H(x_2) &\text{if $t \le s-4$} \\
	  V_1 & \text{\text{if $s-3\le t\le s-2$}}
	\end{cases}.
	  \end{equation*}

We claim that $|Z|\geq (t+1)\varepsilon n$.	Since $(t+2) (2s-2)m +(t+2) \varepsilon n \le 
      \mdeg{Y}{V} \le |V_0| + |V_1| + (2t + 4-2)sm$, we have
      % TNM added n - (per AC's comments)
      \begin{equation} \label{eq:V1}
	|V_1| \ge 2(s - 2 - t)m + (t + 2)\varepsilon n - |V_0|\ge
	(t+1)\varepsilon n.
      \end{equation} 
 So we are done unless $t\leq s-4$.  In this case, note that $|N_H(x_i)| \ge (s - 2)m + \varepsilon n$ for $i \in \{1, 2\}$, which combined with \eqref{eq:V1} gives 	$$|Z|\ge 
	  2(s-2-t)m+(t+1)\varepsilon n -4m\ge (t+1)\varepsilon n.$$
	  %LD
      
      % TNM
     
%      Thus $|Z| \ge (t + 1)\varepsilon n$:
%      \begin{equation*}
%	|Z|\ge 
%	\begin{cases} 
%	  2(s-2-t)m+(t+1)\varepsilon n -4m\ge (t+1)\varepsilon n &\text{if $t \le s-4$} \\
%	  2(s - 2 - t)m + (t + 1)\varepsilon n  & \text{if $s-3\le t\le s-2$}
%	\end{cases}.
%      \end{equation*}
      
      So there  exists $z \in Z\subseteq V_1$ such that
      $(v_1, \dotsc, v_t, z)$ is not useful.  %a $(t+1)$-useful tuple.
      Let $\{y\} = N_L(z) \cap Y$.
      The definitions of useful and $Z$  
      imply $y \notin \{x_1, x_2\}$ and $\mdeg{y}{Y} = 1$.
      But then $\medges{Y - y + z} > \medges{Y}$,
      contradicting the maximality of 
      $\sum_{T \in A} \medges{M[\im(T)]}$.
    \end{proof}
    % TNM

    \begin{proof}[Proof of Theorem~\ref{thmS4}]
      Assume $s \ge 2$ as otherwise the theorem is trivial.
      Let $d := s^2$ and $\alpha := \frac{\varepsilon^d}{2}$.
      For any $S \in \binom{V}{s}$ call 
      $Z \in \binom{V - S}{d}$ an 
      \textit{$S$-sponge}
      if both $M[Z]$ and  $M[Z \cup S]$ have a perfect acceptable 
      $s$-clique tiling.
      Define $f : \binom{V}{s} \to 2^{V^d}$  by
      \begin{equation*}
	f(S) := \{ T \in V^d : \im(T) \text{ is an $S$-sponge} \}.
      \end{equation*}

\begin{figure}[ht]
\begin{center}
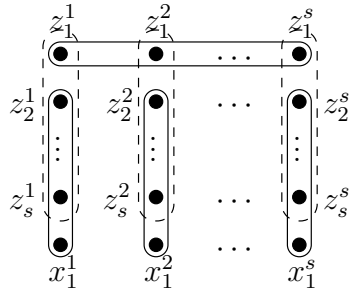
\caption{An $S$-sponge. Note that the tuples indicated by the dashed lines form a tiling and the tuples indicated by the solid lines form a larger tiling.}
\end{center}
\end{figure}

      \begin{unnumbered_claim}
	$|f(S)| \ge \alpha n^d$ for every $S \in \binom{V}{s}$.
      \end{unnumbered_claim}
      \begin{proof}
	Let $S := \{x_1^1, \dotsc, x_1^{s}\} \in \binom{V}{s}$.
	By Lemma~\ref{lem:sponges} there are 
	(many) more than $( \varepsilon n )^{s}$ 
	tuples $T_0 \in V^{s}$
	such that $\im(T_0)$ is an acceptable $s$-clique 
	and $\im(T_0) \cap S = \emptyset$.
	Let $(z_1^1, \dotsc z_1^{s})$ be one such tuple.
	Again by Lemma~\ref{lem:sponges}, 
	for every $i \in [s]$ there are at least $( \varepsilon n )^{s-1}$
	tuples 
	$T_i = (z_2^i,\dots, z_s^i) \in V^{s-1}$ such that
	$x_1^i + \im(T_i)$ and $z_1^i + \im(T_i)$ are both acceptable $s$-cliques.
	Therefore, when $n$ is sufficiently large,
	there are at least $(\varepsilon n)^s((\varepsilon n)^s)^{s-1}\geq \alpha n^d$ tuples 
	$T := (z_1^1,\dots, z_s^1, \dots, z_1^s,\dots, z_s^s)$ such that %HK Something is wrong here.
	if we define $Z := \im(T)$, 
	$Z_0 := \{z_1^1, \dotsc, z_{1}^s\}$ and
	$Z_i := \{z_2^i,\dots, z_s^i\}$ 
	for every $i \in [s]$, then
	\begin{itemize}
	  \item
	    $Z \in \binom{V - S}{d}$;
	  \item
	    $\{ z_1^i + Z_i : i \in [s]\}$ is a perfect 
	    acceptable $s$-clique tiling of $M[Z]$; 
	    and
	  \item
	    $Z_0 + \{ x_1^i + Z_i : i \in [s]\}$ 
	    is a perfect acceptable $s$-clique tiling of $M[Z \cup S]$.
	    \qedhere
	\end{itemize}
	
      \end{proof}

      Let $\gamma, \beta < \min\{\alpha, \frac{\varepsilon}{2}\}$ 
      be constants that
      satisfy the hypothesis of Lemma~\ref{lem:absorbing}, 
      and let $\mathcal{F} \subset V^d$ be a set guaranteed by the lemma.
      Let $Q := \bigcup_{T \in \mathcal{F}} \im(T)$ and note that
      $|Q| = d |\mathcal{F}| < \frac{\varepsilon}{2} n$.
      Let $M' := M - Q$. 

      We now can apply Theorem~\ref{thm:almost_tiling}
      to $M'$ to tile all of the vertices of $M'$ with
      acceptable $s$-cliques except a set
      $X$ of order at most $s(s-1)(2s-1)/3$.
      Partition $X$ into sets of size $s$.
      If $n$ is sufficiently large, $|X| \le s \gamma n$.
      Therefore, 
      for every set $S$ in the partition of $X$, 
      we can choose a unique $T \in f(S) \cap \mathcal{F}$.
      This implies that there is a perfect acceptable $s$-clique tiling
      of $M[X \cup Q]$ which completes the proof.
    \end{proof}

    \hbadness 10000\relax


\begin{thebibliography}{1}
	\bibitem{chernoff1952measure} H.~Chernoff, A measure of asymptotic
	efficiency for tests of a hypothesis based on the sum of observations,
	\emph{The Annals of Mathematical Statistics} \textbf{23} (1952), no.~4,
	493--507.

	\bibitem{corradi1963maximal} K.~Corr{á}di and A.~Hajnal, On
	the maximal number of independent circuits in a graph, 
	\emph{Acta Mathematica Hungarica} \textbf{14} (1963), no.~3, 423--439.

	\bibitem{CKM12} A. Czygrinow, H.A. Kierstead and T. Molla,
	On directed versions of the Corr\'{a}di-Hajnal Corollary,
	submitted.

	\bibitem{dirac1952some} G.A. Dirac, Some theorems on abstract
	graphs, \emph{Proceedings of the London Mathematical Society} \textbf{3}
	(1952), no.~1, 69--81.

	\bibitem{E} Enomoto,~H., 
	On the existence of disjoint cycles in a graph, \emph{Combinatorica},
	\textbf{18}, (1998), no.~4, 487--492.

	\bibitem{enomoto87} Enomoto,~H., Kaneko,~A. and Tuza,~ Z.,
	$P_3$-factors and covering cycles in graphs of minimum degree $n/3$, 
	\emph{Combinatorics}, \textbf{52} (1987), 213--220.

	\bibitem{erdos}
	Erd\H{o}s,~P. Problem 9. In \emph{Theory of Graphs and its Applications}
	Czech.  Academy of Sciences, Prague, (1963), 159

	\bibitem{gh}A. Ghouila-Houri, 
	Une condition suffisante d'existence d'un circuit hamiltonien,
	\emph{C. R. Math. Acad. Sci. Paris}, \textbf{25}, (1960), 495--497.

	\bibitem{hajnal1970pcp} A.~Hajnal and E.~Szemer{é}di, Proof
	of a conjecture of {P}. {E}rd{\H{o}}s, \emph{Combinatorial Theory
	and Its Application} \textbf{2} (1970), 601--623.
      \bibitem{ht00}F. Havet, S. Thomass\'{e}, Oriented hamiltonian paths in
	tournaments: A proof of Rosenfeld's conjecture, \emph{Journal
	Combinatorial Theory B} \textbf{78}, (2000), 243--273.
	\bibitem{ks}P. Keevash, B. Sudakov, Triangle packings and 1-factors in oriented graphs, \emph{Journal
	Combinatorial Theory B} \textbf{99}, (2009), 709--727.
	\bibitem{kko}P. Keevash, D. Kuhn and D. Osthus, An exact minimum degree condition for Hamilton
	cycles in oriented graphs, \emph{J. London Math. Soc.} \textbf{79}, (2009), 144--166.

	\bibitem{KK}H. A. Kierstead and A. V. Kostochka, A Short Proof of
	the Hajnal-Szemer\' edi Theorem on equitable coloring, \emph{Combinatorics,
	Probability and Computing}, 17 (2008) 265--270.

	\bibitem{KKore}H. A. Kierstead and A. V. Kostochka, An Ore-type theorem
	on equitable coloring, \emph{J. Combinatorial Theory Series B}, 98
	(2008) 226--234.

	\bibitem{KKMS}H. A. Kierstead, A. V. Kostochka, M. Mydlarz and E.
	Szemer\' edi, A fast algorithm for equitable coloring, \emph{Combinatorica},
	{30} (2010) 217--224. %HK update

	\bibitem{KKY}H. A. Kierstead, A. V. Kostochka and Gexin Yu, Extremal
	graph packing problems: Ore-type versus Dirac-type, in \emph{Surveys
	in Combinatorics 2009} (eds. S. Huczynska, J. Mitchell and C. Roney-Dougal),
	London Mathematical Society Lecture Note Series \textbf{365}, Cambridge
	University Press, Cambridge (2009) 113--136.

	\bibitem{komlos1996square} J.~Koml{ó}s, G.N. S{á}rk{ö}zy,
	and E.~Szemer{é}di, On the square of a hamiltonian cycle
	in dense graphs, \emph{Random Structures \& Algorithms} \textbf{9} (1996),
	no.~1-2, 193--211.

	\bibitem{KMOsumner11} D.~K{\"u}hn, R.~Mycroft, D.~Osthus, A
	proof of Sumner's universal tournament conjecture for large
	tournaments, \emph{Proceedings of the London Mathematical Society} \textbf{4} (2011), 731--766.


	\bibitem{levitt2010avoid} I.~Levitt, G.N. S{á}rk{ö}zy, and E.~Szemer{é}di,
	How to avoid using the regularity lemma: P{ó}sa's conjecture
	revisited, \emph{Discrete Mathematics} \textbf{310} (2010), no.~3, 630--641.

	\bibitem{nash1971edge} C.St.J.A. Nash-Williams, Edge-disjoint
	hamiltonian circuits in graphs with vertices of large valency, \emph{Studies
	in Pure Mathematics} (Presented to Richard Rado), 1971, pp.~157--183.

	\bibitem{wangdir}H. Wang, Independent Directed Triangles in Directed Graphs,
	\emph{Graphs and Combinatorics}, \textbf{16}, (2000), 453--462.

	\bibitem{wood}D. Woodall, Sufficient conditions for cycles in digraphs, \emph{Proc. London Math. Soc.}, \textbf{24}, (1972), 739--755.
    \end{thebibliography}
    \end{document}